\documentclass[11pt]{amsart}
\usepackage[english]{babel}
\usepackage{lipsum}
\usepackage{mathrsfs}
\usepackage{stix}
\usepackage{fullpage}
\usepackage{mathtools}
\usepackage{amsmath}
\usepackage{tikz-cd}
\usepackage[colorlinks]{hyperref}
\usepackage{fullpage}
\usepackage{mathtools}
\usepackage{amsmath}

\newtheorem{thm}{Theorem}[section]
\newtheorem{lem}[thm]{Lemma}

\theoremstyle{definition}

\theoremstyle{remark}
\newtheorem{rem}{Remark}[section]
\newtheorem{defn}{Definition}

\numberwithin{equation}{section}

\begin{document}

\title[Fractional Schr\"{o}dinger Equation with singularity]{Fractional Schr\"{o}dinger Equation with singular potentials of higher-order}

\author[A. Altybay]{Arshyn Altybay}
\address{
  Arshyn Altybay:
    \endgraf
    Institute of Mathematics and Mathematical Modeling
    \endgraf
    Al-Farabi Kazakh National University
    \endgraf
    Almaty, Kazakhstan
    \endgraf
    and 
    \endgraf
    Department of Mathematics: Analysis, Logic and Discrete Mathematics
    \endgraf
    Ghent University, Belgium
    \endgraf
    {\it E-mail address} {\rm arshyn.altybay@gmail.com}
    }

\author[M. Ruzhansky]{Michael Ruzhansky}
\address{
  Michael Ruzhansky:
  \endgraf
  Department of Mathematics: Analysis, Logic and Discrete Mathematics
  \endgraf
  Ghent University, Belgium
  \endgraf
  and
  \endgraf
  School of Mathematical Sciences
  \endgraf
  Queen Mary University of London
  \endgraf
  United Kingdom
  \endgraf
  {\it E-mail address} {\rm michael.ruzhansky@ugent.be}
}

\author[M. Sebih]{Mohammed Elamine Sebih}
\address{
  Mohammed Elamine Sebih:
  \endgraf
  Laboratory of Analysis and Control of Partial Differential Equations 
  \endgraf
  Djillali Liabes University 
  \endgraf
  Sidi Bel Abbes, Algeria
  \endgraf
   and  
  \endgraf
  Laboratory of Geomatics, Ecology and Environment
  \endgraf
  University Mustapha Stambouli of Mascara
  \endgraf
  Algeria
  \endgraf
  {\it E-mail address} {\rm sebihmed@gmail.com, ma.sebih@univ-mascara.dz}
}

\author[N. Tokmagambetov]{Niyaz Tokmagambetov}
\address{
  Niyaz Tokmagambetov:
  \endgraf
  Department of Mathematics: Analysis, Logic and Discrete Mathematics
  \endgraf
  Ghent University, Belgium
  \endgraf
   and
  \endgraf
  Al-Farabi Kazakh National University
  \endgraf
  Almaty, Kazakhstan
  \endgraf
  {\it E-mail address} {\rm tokmagambetov@math.kz}
 }



\begin{abstract}
In this paper the space-fractional Schr\"{o}dinger equations with  singular potentials are studied. Delta like or even higher-order singularities are allowed. By using the regularising techniques, we introduce a family of 'weakened' solutions, calling them very weak solutions. The existence, uniqueness and consistency results are proved in an appropriate sense. Numerical simulations are done, and a particles accumulating effect is observed in the singular cases. From the mathematical point of view a "splitting of the strong singularity" phenomena is also observed.
\end{abstract}

\maketitle

\section{Introduction}
In this paper we investigate the fractional Schr\"{o}dinger equation with distributional potentials. Namely, the following Cauchy problem
\begin{equation}
    \left\lbrace
    \begin{array}{l}
    iu_{t}(t,x)+(-\Delta)^{s} u(t,x) + p(x)u(t,x)=0 ,~~~(t,x)\in\left(0,T\right)\times \mathbb{R}^{d},\\
    u(0,x)=u_{0}(x), \label{Equation in introduction}
    \end{array}
    \right.
\end{equation}
is a subject to our investigation. Here $p$ is assumed to be non-negative, and $s>0$. We consider the fractional Laplacian as a spacial operator instead of the classical one and prove that the problem has a so-called "very weak solution".

While the study of the fractional Schr\"{o}dinger equation is mathematically challenging, from the physical point of view it is a natural extension of the standard Schr\"{o}dinger equation when the Brownian trajectories in Feynman path integrals are replaced by Levy flights. The fractional Schr\"{o}dinger equation is introduced by Laskin in quantum mechanics \cite{Las00}, \cite{Las02}. More recently, it is proposed as a model in optics by Longhi \cite{Lon15} and applied to laser implementation. For more general overview about the fractional Schr\"{o}dinger equation and its related topics in physics, one can see \cite{GX06,Las18}. In recent years, it has attracted a lot of interest by many authors, for instance in \cite{RB10, CHHO13, LK16, Ela19}.

On the other hand, our intention to consider singular potentials is also natural from a physical point of view. It can describe a particle which is free to move in two regions of space with a barrier between the two regions. For example, an electron can move almost freely in a conducting material, but if two conducting surfaces are put close together, the interface between them acts as a barrier for the electron. The fractional Schr\"{o}dinger equation with singular potentials has been also investigated by many authors, we cite for instance \cite{OCV10, OV11, LRSRM13, DGV18, GV19, AS19} and the references mentioned there.

As mentioned before, our aim is to prove the well-posedness of the Cauchy problem (\ref{Equation in introduction}) in the very weak sense. The appearance of the concept of very weak solutions traces back to the paper \cite{GR15}, where the authors introduced the concept for the analysis of second order hyperbolic equations with irregular coefficients in time. It was later applied in \cite{MRT19}, \cite{RT17a}, and \cite{RT17b} to study the different physical models, and some numerical analysis was done in \cite{ART19}. All these works deal with equations with time-dependent coefficients. We want here to apply it for the fractional Schr\"{o}dinger equation with space-depending coefficients. Note that in the recent paper \cite{Gar20}, the author starts to study very weak solutions for the wave equation with spatial variable coefficients. Very weak solutions for fractional Klein-Gordon equations with singular masses were considered in \cite{ARST21}.

The novelty of our work is two-fold. On the one hand, we apply the concept of very weak solutions which seems to be well adapted for numerical simulations. Indeed, we can talk about uniqueness of numerical solutions in some appropriate sense. On the other hand, the use of the theory of very weak solutions allows us to overcome the celebrated problem of the impossibility of multiplication of distributions \cite{Sch54} and, thus, to consider general coefficients: distributional or non-distributional.

In the present paper we will use the following notations :
\begin{itemize}
    \item $f\lesssim g$ means that there exists a positive constant $C$ such that $f \leq Cg$;
    \item for $s>0$, the fractional Sobolev space $H^{s}(\mathbb{R}^{d})$ is defined as:
\begin{equation*}
    H^{s}(\mathbb{R}^{d})=\big\{ u\in L^{2}(\mathbb{R}^{d}): \Vert u\Vert_{H^{s}}:=\Vert u\Vert_{L^2}+\Vert (-\Delta)^{\frac{s}{2}}u\Vert_{L^2} < +\infty \big\};
\end{equation*}
    \item for $k\in\mathbb Z_{+}$, we denote by $\Vert \cdot\Vert_{k}$ the norm defined by
    \begin{equation*}
        \Vert u(t,\cdot)\Vert_{k}:=\sum_{l=0}^{k}\Vert \partial_{t}^{l}u(t,\cdot)\Vert_{L^2} + \Vert (-\Delta)^{\frac{s}{2}}u(t,\cdot)\Vert_{L^2},
    \end{equation*}
and simply denote it by $\Vert u(t, \cdot)\Vert$, when $k=0$. We note that $\Vert u(t, \cdot)\Vert\cong\Vert u(t, \cdot)\Vert_{H^{s}}$.
\end{itemize}

\section{Main results}

In this section we introduce a family of 'very weak solutions' for the space-fractional Schr\"{o}dinger equations with distributional potentials. In particular, we are interested in singularities of delta-like or even higher-order types. First, we start by considering a classical case. Later, we show the existence, uniqueness and consistency of the very weak solutions in an appropriate sense.

Let us fix $T>0$. For a positive $s$, we consider the initial problem for the space-fractional Schr\"{o}dinger equation
\begin{equation}
    \left\lbrace
    \begin{array}{l}
    iu_{t}(t,x)+(-\Delta)^{s} u(t,x) + p(x)u(t,x)=0 ,~~~(t,x)\in\left(0,T\right)\times \mathbb{R}^{d},\\
    u(0,x)=u_{0}(x), \label{Equation}
    \end{array}
    \right.
\end{equation}
where the potential $p$ is non-negative and can be singular. 

We start by stating the following result dealing with the case of regular enough coefficient $p$.
\begin{lem} 
\label{Lem energy estimates}
Let $s>0$. Suppose that $p\in L^{\infty}(\mathbb{R}^{d})$ be non-negative and assume that $u_{0}\in H^{s}(\mathbb{R}^{d})$. Then the estimate
\begin{equation}
    \Vert u(t, \cdot)\Vert_{H^{s}} \lesssim \left(1+\Vert p\Vert_{L^{\infty}}\right)\Vert u_{0}\Vert_{H^{s}}, \label{Energy estimate}
\end{equation}
holds for the unique solution $u\in C(\left[0,T\right];H^{s})$ to the Cauchy problem (\ref{Equation}).
\end{lem}

\begin{proof}
We multiply the equation in (\ref{Equation}) by $u_t$ and by integrating, we get
\begin{equation}
    Re \left(\langle i\partial_{t}u(t,\cdot),\partial_{t}u(t,\cdot)\rangle_{L^2} + \langle(-\Delta)^{s}u(t,\cdot),\partial_{t}u(t,\cdot)\rangle_{L^2} + \langle p(\cdot)u(t,\cdot),\partial_{t}u(t,\cdot)\rangle_{L^2} \right)=0. \label{Energy functional}
\end{equation}
It is easy to see that
\begin{equation*}
    Re \langle i\partial_{t}u(t,\cdot),\partial_{t}u(t,\cdot)\rangle_{L^2}=0,
\end{equation*}
\begin{equation*}
    Re \langle p(\cdot)u(t,\cdot),\partial_{t}u(t,\cdot)\rangle_{L^2}=\frac{1}{2}\partial_{t}\Vert p^{\frac{1}{2}}(\cdot)u(t,\cdot)\Vert_{L^2}^{2},
\end{equation*}
and
\begin{equation*}
    Re \langle(-\Delta)^{s}u(t,\cdot),\partial_{t}u(t,\cdot)\rangle_{L^2}=\frac{1}{2}\partial_{t}\Vert (-\Delta)^{\frac{s}{2}}u(t,\cdot)\Vert_{L^2}^{2}.
\end{equation*}
The last equality is a consequence of the fact that $(-\Delta)^{s}$ is a self-adjoint operator. Let us denote by
\begin{equation*}
    E(t):=\Vert (-\Delta)^{\frac{s}{2}}u(t,\cdot)\Vert_{L^2}^{2} + \Vert p^{\frac{1}{2}}(\cdot)u(t,\cdot)\Vert_{L^2}^{2}.
\end{equation*}
It follows from (\ref{Energy functional}) that $\partial_{t}E(t)=0$ and thus
\begin{equation*}
    E(t)=E(0).
\end{equation*}
Therefore
\begin{equation}
    \Vert p^{\frac{1}{2}}u(t,\cdot)\Vert_{L^2}^{2} \lesssim \Vert (-\Delta)^{\frac{s}{2}}u_{0}\Vert_{L^2}^{2} + \Vert p\Vert_{L^{\infty}} \Vert u_0\Vert_{L^2}^{2} \label{Estimate pxu}
\end{equation}
and
\begin{equation*}
    \Vert (-\Delta)^{\frac{s}{2}}u(t,\cdot)\Vert_{L^2}^{2} \lesssim \Vert (-\Delta)^{\frac{s}{2}}u_{0}\Vert_{L^2}^{2} + \Vert p\Vert_{L^{\infty}} \Vert u_0\Vert_{L^2}^{2},
\end{equation*}
where we use that $\Vert p^{\frac{1}{2}} \, u_{0}\Vert_{L^2}^{2}$ can be estimated by
\begin{equation*}
    \Vert p^{\frac{1}{2}}\, u_{0} \Vert_{L^2}^{2}\leq \Vert p \, \Vert_{L^{\infty}}\Vert u_{0} \Vert_{L^2}^{2}.
\end{equation*}
Moreover, it follows that
\begin{equation}
    \Vert (-\Delta)^{\frac{s}{2}}u(t,\cdot)\Vert_{L^2} \lesssim \left(1+\Vert p\Vert_{L^{\infty}}^{\frac{1}{2}}\right)\Vert u_{0}\Vert_{H^{s}}. \label{Estimate nabla u}
\end{equation}
Let us estimate $u$. After application of the Fourier transformation in (\ref{Equation}), we get the auxiliary Cauchy problem
\begin{equation}
    i\hat{u}_{t}(t,\xi)+\vert \xi\vert^{2s}\hat{u}(t,\xi)=\hat{f}(t,\xi);~~\hat{u}(0,\xi)=\hat{u}_{0}(\xi), \label{Equation Fourier}
\end{equation}
where $\hat{u}$, $\hat{f}$ denote the Fourier transforms of $u$ and $f$ with respect to the spacial variable $x$ and $f(t,x):=-p(x)u(t,x)$. Using Duhamel's principle (see, e.g. \cite{ER18}), we get the following representation of the solution to the Cauchy problem (\ref{Equation Fourier})
\begin{equation}
    \hat{u}(t,\xi)=\hat{u}_{0}(\xi)\exp(-i\vert\xi\vert^{2s}t) + \int_{0}^{t}\exp\big(-i\vert\xi\vert^{2s}(t-s)\big)\hat{f}(s,\xi) ds. \label{Representation of sol}
\end{equation}
Taking the $L^2$ norm in (\ref{Representation of sol}) and using the fact that $\exp(-i\vert\xi\vert^{2s}t)$ is a unitary operator, we get the estimate
\begin{equation}
    \Vert \hat{u}(t,\cdot)\Vert_{L^2} \lesssim \Vert \hat{u}_{0}\Vert_{L^2} + \int_{0}^{T}\Vert \hat{f}(s,\cdot)\Vert_{L^2} ds.
\end{equation}
Using the Plancherel-Parseval formula, the estimate (\ref{Estimate pxu}) and the fact that $\Vert f(t,\cdot)\Vert_{L^2}=\Vert p(\cdot)u(t,\cdot)\Vert_{L^2}$ can be estimated by
\begin{equation*}
    \Vert p(\cdot)u(t,\cdot)\Vert_{L^{2}} \leq \Vert p \Vert_{L^{\infty}}^{\frac{1}{2}}\Vert p^{\frac{1}{2}}u(t,\cdot)\Vert_{L^{2}},
\end{equation*}
we arrive at
\begin{equation}
    \Vert u(t,\cdot)\Vert_{L^2} \lesssim \left(1+\Vert p\Vert_{L^{\infty}}^{\frac{1}{2}}\right)^{2}\Vert u_{0}\Vert_{H^{s}}. \label{Estimate u in proof}
\end{equation}
By summing (\ref{Estimate nabla u}) and (\ref{Estimate u in proof}), we get our estimate. Thus, the lemma is proved.
\end{proof}

\begin{rem}
Requiring further regularity on the initial data $u_0$, one can prove that the estimate
\begin{equation*}
    \Vert u(t,\cdot)\Vert_{k} \lesssim \left(1+\Vert p\Vert_{L^{\infty}}\right)\Vert u_{0}\Vert_{H^{s(1+2k)}},
\end{equation*}
holds for all $k\geq 0$. For this, we use the estimate (\ref{Estimate u in proof}) and proceed by induction on $k\geq 1$, on the property that, if $v_{k}:=\partial_{t}^{k}u$, where $u$ is the solution to the Cauchy problem (\ref{Equation}), solves the equation
\begin{equation*}
    i\partial_{t}v_{k}(t,x)+(-\Delta)^{s} v_{k}(t,x) + p(x)v_{k}(t,x)=0,
\end{equation*}
with initial data $v_{k}(0,x)$, then $v_{k+1}=\partial_{t}v_{k}$ solves the same equation with initial data
\begin{equation*}
    v_{k+1}(0,x)=-i(-\Delta)^{s} v_{k}(0,x)-ip(x)v_{k}(0,x).
\end{equation*}
\end{rem}

\subsection{Existence of very weak solutions}
In what follows, we consider the case when the potential $p$ is strongly singular, we have in mind  $\delta$ or $\delta^2$-functions. As mentioned above, we want to prove the existence of a very weak solution to the Cauchy problem (\ref{Equation}). We first regularise the coefficient $p$ and the data $u_0$ by convolution with a suitable mollifier $\psi$ and  obtain families of smooth functions $(p_{\varepsilon})_{\varepsilon}$ and $(u_{0,\varepsilon})_{\varepsilon}$, namely
\begin{equation*}
    p_{\varepsilon}(x)=p\ast \psi_{\varepsilon }(x) ~~~ and~~u_{0,\varepsilon}(x)=u_{0}\ast \psi_{\varepsilon }(x),
\end{equation*}
where
\begin{equation*}
    \psi_{\varepsilon }(x)=\varepsilon^{-d}\psi(x/\varepsilon),~~~\varepsilon\in\left(0,1\right],
\end{equation*}
and the function $\psi$ is a Friedrichs-mollifier, i.e. $\psi\in C_{0}^{\infty}(\mathbb{R}^{d})$, $\psi\geq 0$ and $\int\psi =1$. The above regularisation works when $p$ is at least a distribution. For more generality, we will make assumptions on the regularisations $(p_{\varepsilon})_{\varepsilon}$ and $(u_{0,\varepsilon})_{\varepsilon}$, instead of making them on $p$ and $u_{0}$. That is, we assume that there exist $N, N_{0}\in \mathbb{N}_{0}$ such that
\begin{equation}
    \Vert p_{\varepsilon}\Vert_{L^{\infty}}\leq C\varepsilon^{-N} \label{Moderateness assumption coeff}
\end{equation}
and
\begin{equation}
    \Vert u_{0,\varepsilon}\Vert_{H^s}\leq C_{0}\varepsilon^{-N_0}. \label{Moderateness assumption data}
\end{equation}
We have the following definition.

\begin{defn}[Moderateness] \label{Def:Moderetness}
\leavevmode
\begin{itemize}
    \item[(i)] We say that the net of functions $(f_{\varepsilon})_{\varepsilon}$ is ${H^{s}}$-moderate, if there exist $N\in\mathbb{N}_{0}$ and $c>0$ such that
\begin{equation*}
    \Vert f_{\varepsilon}\Vert_{H^{s}} \leq c\varepsilon^{-N}.
\end{equation*}
    \item[(ii)] We say that the net of functions $(g_{\varepsilon})_{\varepsilon}$ is $L^{{\infty}}$-moderate, if there exist $N\in\mathbb{N}_{0}$ and $c>0$ such that
\begin{equation*}
    \Vert g_{\varepsilon}\Vert_{L^{\infty}} \leq c\varepsilon^{-N}.
\end{equation*}
    \item[(iii)] We say that the net of functions $(u_{\varepsilon})_{\varepsilon}$ from $C(\left[0,T\right];H^{s})$ is $H^{s}$-moderate, if there exist $N\in\mathbb{N}_{0}$ and $c>0$ such that
\begin{equation*}
   \Vert u_{\varepsilon}(t, \cdot)\Vert_{H^{s}} \leq c\varepsilon^{-N}
\end{equation*}
for all $t\in[0,T]$.
\end{itemize}
\end{defn}

\begin{rem}
We see that $(u_{0,\varepsilon})_{\varepsilon}$ and $(p_{\varepsilon})_{\varepsilon}$ are moderate by assumptions. We also note that such assumptions are natural for distributional coefficients in the sense that regularisations of distributions are moderate. Precisely, by the structure theorems for distributions (see, e.g. \cite{FJ98}), we know that
\begin{equation}
    \text{Compactly supported distributions~} \mathcal{E}^{\prime}(\mathbb{R}^{d})\subset \big\{C^{\infty}(\mathbb{R}^{d})-\text{moderate families}\big\},\label{Strucure thm}
\end{equation}
and we see from (\ref{Strucure thm}), that a solution to a Cauchy problem may not exist in the sense of distributions, while it may exist in the set of $C^{\infty}$-moderate functions.
\end{rem}
Now, let us introduce the notion of a very weak solution to the Cauchy problem (\ref{Equation}).

\begin{defn}[Very weak solution]
The net $(u_{\varepsilon})_{\varepsilon}\in C(\left[0,T\right];H^{s})$ is said to be a very weak solution of order $s$ to the Cauchy problem (\ref{Equation}) if there exist an ${L^{\infty}}$-moderate regularisation of the coefficient $p$ and $H^s$-moderate regularisation of $u_0$ such that $(u_{\varepsilon})_{\varepsilon}$ solves the regularized problem
\begin{equation}
    \left\lbrace
    \begin{array}{l}
    i\partial_{t}u_{\varepsilon}(t,x)+(-\Delta)^{s} u_{\varepsilon}(t,x) + p_{\varepsilon}(x)u_{\varepsilon}(t,x)=0 ,~~~(t,x)\in\left[0,T\right]\times \mathbb{R}^{d},\\
    u_{\varepsilon}(0,x)=u_{0,\varepsilon}(x), \label{Regularized equation}
    \end{array}
    \right.
\end{equation}
for all $\varepsilon\in\left(0,1\right]$, and is $C$-moderate.
\end{defn}
We can now state the following theorem, the proof of which follows immediately from the definitions. In what follows we understand $p\geq0$ as its regularisations $p_{\varepsilon}$ satisfy $p_{\varepsilon}\geq0$ for all $\varepsilon\in(0, 1]$. This is clearly the case when $p$ is a distribution.

\begin{thm}[Existence] \label{Thm Existence}
Let $p\geq 0$ and $s>0$. Assume that the regularisations of the coefficient $p$ and the Cauchy data $u_{0}$ satisfy the assumptions (\ref{Moderateness assumption coeff}) and (\ref{Moderateness assumption data}). Then the Cauchy problem (\ref{Equation}) has a very weak solution.
\end{thm}

\begin{proof}
The coefficient $p$ and the data $u_0$ are moderate by assumptions. To prove that a very weak solution exists, we need to prove that the net $(u_{\varepsilon})_{\varepsilon}$, solution to the family of regularized Cauchy problems
\begin{equation*}
    \left\lbrace
    \begin{array}{l}
    i\partial_{t}u_{\varepsilon}(t,x)+(-\Delta)^{s} u_{\varepsilon}(t,x) + p_{\varepsilon}(x)u_{\varepsilon}(t,x)=0 ,~~~(t,x)\in\left[0,T\right]\times \mathbb{R}^{d},\\
    u_{\varepsilon}(0,x)=u_{0,\varepsilon}(x),
    \end{array}
    \right.
\end{equation*}
is $C$-moderate. Indeed, using the assumptions (\ref{Moderateness assumption coeff}), (\ref{Moderateness assumption data}) and the energy estimate (\ref{Energy estimate}), we arrive at
\begin{equation*}
    \Vert u(t,\cdot)\Vert \lesssim \varepsilon^{-N_{0}-N}.
\end{equation*}
The net $(u_{\varepsilon})_{\varepsilon}$ is then $C$-moderate and the existence of a very weak solution is proved.
\end{proof}

In order to prove uniqueness and consistency of the very weak solution in the forthcoming theorems, we need the following lemma.

\begin{lem}\label{Lem energy estimate 2}
Let $u_{0}\in H^{s}(\mathbb{R}^{d})$ and assume that $p\in L^{\infty}(\mathbb{R}^d)$ is non-negative. Then, the energy conservation
\begin{equation}
    \Vert u(t,\cdot)\Vert_{L^2} = \Vert u_{0}\Vert_{L^2}, \label{Energy estimate 2}
\end{equation}
holds for all $t\in [0,T]$, for the unique solution $u\in C(\left[0,T\right];H^{s})$ to the Cauchy problem (\ref{Equation}).
\end{lem}

\begin{proof}
We first multiply the equation in (\ref{Equation}) by $-i$, we obtain
\begin{equation*}
    u_{t}(t,x)-i(-\Delta)^{s} u(t,x) - ip(x)u(t,x)=0.
\end{equation*}
Multiplying the last equation by $u$, integrating over $\mathbb{R}^{d}$ and taking the real part, we get
\begin{equation*}
    Re \left(\langle u_{t}(t,\cdot),u(t,\cdot)\rangle_{L^2} -i \langle (-\Delta)^{s} u(t,\cdot),u(t,\cdot)\rangle_{L^2} -i \langle p(\cdot)u(t,\cdot),u(t,\cdot)\rangle_{L^2} \right)=0.
\end{equation*}
Using similar arguments as in Lemma \ref{Lem energy estimates}, it is easy to see that
\begin{equation*}
    Re \langle u_{t}(t,\cdot),u(t,\cdot)\rangle_{L^2} = \frac{1}{2}\partial_{t} \Vert u(t,\cdot)\Vert_{L^2}^{2}
\end{equation*}
and that
\begin{equation*}
    Re \left(-i \langle (-\Delta)^{s} u(t,\cdot),u(t,\cdot)\rangle_{L^2}\right) = Re \left(-i \langle p(\cdot)u(t,\cdot),u(t,\cdot)\rangle_{L^2}\right) = 0.
\end{equation*}
Thus, we have the energy conservation law, i.e. $\Vert u(t,\cdot)\Vert_{L^2}$ is constant for all $t\in [0,T]$ and the statement is proved.
\end{proof}

\subsection{Uniqueness}
We prove the uniqueness of a very weak solution to the Cauchy problem (\ref{Equation}) in the sense of the following definition.

\begin{defn}[Uniqueness]
We say that the Cauchy problem (\ref{Equation}) has a unique very weak solution, if for all families of regularisations $(p_{\varepsilon})_{\varepsilon}$, $(\Tilde{p}_{\varepsilon})_{\varepsilon}$, $(u_{0,\varepsilon})_{\varepsilon}$ and $(\Tilde{u}_{0,\varepsilon})_{\varepsilon}$ of the coefficient $p$ and the Cauchy data $u_0$, satisfying
\begin{equation*}
    \Vert p_{\varepsilon}-\Tilde{p}_{\varepsilon}\Vert_{L^{\infty}}\leq C_{k}\varepsilon^{k} \text{~~for all~~} k>0
\end{equation*}
and
\begin{equation*}
    \Vert u_{0,\varepsilon}-\Tilde{u}_{0,\varepsilon}\Vert_{L^{2}}\leq C_{l}\varepsilon^{l} \text{~~for all~~} l>0,
\end{equation*}
we have
\begin{equation*}
    \Vert u_{\varepsilon}(t,\cdot)-\Tilde{u}_{\varepsilon}(t,\cdot)\Vert_{L^{2}} \leq C_{N}\varepsilon^{N}
\end{equation*}
for all $N>0$,  
where $(u_{\varepsilon})_{\varepsilon}$ and $(\Tilde{u}_{\varepsilon})_{\varepsilon}$ are the families of solutions to the corresponding regularized Cauchy problems.
\end{defn}

\begin{thm}[Uniqueness] \label{Thm uniqueness}
Let $p\geq 0$ and $u_{0}\in H^{s}(\mathbb{R}^{d})$ and assume that they satisfy the assumptions (\ref{Moderateness assumption coeff}) and (\ref{Moderateness assumption data}). Then, the Cauchy problem (\ref{Equation}) has a unique very weak solution.
\end{thm}

\begin{proof}
Let $(p_{\varepsilon})_{\varepsilon}$, $(\Tilde{p}_{\varepsilon})_{\varepsilon}$ and $(u_{0,\varepsilon})_{\varepsilon}$, $(\Tilde{u}_{0,\varepsilon})_{\varepsilon}$, regularisations of $p$ and $u_0$, satisfying
\begin{equation*}
    \Vert p_{\varepsilon}-\Tilde{p}_{\varepsilon}\Vert_{L^{\infty}}\leq C_{k}\varepsilon^{k} \text{~~for all~~} k>0
\end{equation*}
and
\begin{equation*}
    \Vert u_{0,\varepsilon}-\Tilde{u}_{0,\varepsilon}\Vert_{L^{2}}\leq C_{l}\varepsilon^{l} \text{~~for all~~} l>0,
\end{equation*}
and let us denote by $U_{\varepsilon}(t,x):=u_{\varepsilon}(t,x)-\Tilde{u}_{\varepsilon}(t,x)$, where $(u_{\varepsilon})_{\varepsilon}$ and $(\Tilde{u}_{\varepsilon})_{\varepsilon}$ are the families of solutions to the regularized Cauchy problems, corresponding to the families $\left(p_{\varepsilon}, u_{\varepsilon}\right)_{\varepsilon}$ and $\left(\Tilde{p}_{\varepsilon}, \Tilde{u}_{0,\varepsilon}\right)_{\varepsilon}$. Then, $U_{\varepsilon}$ solves the Cauchy problem
\begin{equation}
    \left\lbrace
    \begin{array}{l}
    i\partial_{t}U_{\varepsilon}(t,x)+(-\Delta)^{s} U_{\varepsilon}(t,x) + p_{\varepsilon}(x)U_{\varepsilon}(t,x)=f_{\varepsilon}(t,x) ,~~~(t,x)\in\left(0,T\right)\times \mathbb{R}^{d},\\
    U_{\varepsilon}(0,x)=(u_{0,\varepsilon}-\Tilde{u}_{0,\varepsilon})(x), \label{Cauchy pb U_eps}
    \end{array}
    \right.
\end{equation}
where
\begin{equation*}
    f_{\varepsilon}(t,x)=\left(\Tilde{p}_{\varepsilon}(x)-p_{\varepsilon}(x)\right)\Tilde{u}_{\varepsilon}(t,x).
\end{equation*}
Let $(V_{\varepsilon})_{\varepsilon}$ and $(W_{\varepsilon})_{\varepsilon}$, the families of solutions to the auxiliary Cauchy problems
\begin{equation*}
    \left\lbrace
    \begin{array}{l}
    i\partial_{t}V_{\varepsilon}(x,t;s) + (-\Delta)^{s}V_{\varepsilon}(x,t;s) + p_{\varepsilon}(x)V_{\varepsilon}(x,t;s)=0,\\
    V_{\varepsilon}(x,s;s)=f_{\varepsilon}(s,x),
    \end{array}
    \right.
\end{equation*}
and
\begin{equation*}
    \left\lbrace
    \begin{array}{l}
    i\partial_{t}W_{\varepsilon}(t, x) + (-\Delta)^{s}W_{\varepsilon}(t, x) + p_{\varepsilon}(x)W_{\varepsilon}(t, x)=0,\\
    W_{\varepsilon}(0, x)=(u_{0,\varepsilon}-\Tilde{u}_{0,\varepsilon})(x).
    \end{array}
    \right.
\end{equation*}
Using Duhamel's principle, $U_{\varepsilon}$ is given by
\begin{equation}
    U_{\varepsilon}(t, x)=W_{\varepsilon}(t, x) + \int_{0}^{t}V_{\varepsilon}(x,t-s;s)ds. \label{Representation U_eps}
\end{equation}
Taking the $L^2$ norm in (\ref{Representation U_eps}) and using (\ref{Energy estimate 2}) to estimate $V_{\varepsilon}$ and $W_{\varepsilon}$, we get
\begin{align*}
    \Vert U_{\varepsilon}(t, \cdot)\Vert_{L^2} & \leq \Vert W_{\varepsilon}(t, \cdot)\Vert_{L^2} + \int_{0}^{T}\Vert V_{\varepsilon}(\cdot,t-s;s)\Vert_{L^2} ds\\
    & \lesssim \Vert u_{0,\varepsilon}-\Tilde{u}_{0,\varepsilon}\Vert_{L^2} + \int_{0}^{T}\Vert f_{\varepsilon}(s,\cdot)\Vert_{L^2} ds\\
    & \lesssim \Vert u_{0,\varepsilon}-\Tilde{u}_{0,\varepsilon}\Vert_{L^2} + \Vert \Tilde{p}_{\varepsilon}-p_{\varepsilon}\Vert_{L^{\infty}}\int_{0}^{T}\Vert \Tilde{u}_{\varepsilon}(s,\cdot)\Vert_{L^2} ds.
\end{align*}
From the one hand, we have that $\Vert u_{0,\varepsilon}-\Tilde{u}_{0,\varepsilon}\Vert_{L^{2}}\leq C_{l}\varepsilon^{l}$, for all $l>0$. On the other hand, $(u_{\varepsilon})_{\varepsilon}$ as a very weak solution to the Cauchy problem (\ref{Equation}) is moderate and $\Vert \Tilde{p}_{\varepsilon}-p_{\varepsilon}\Vert_{L^{\infty}} \leq C_{k}\varepsilon^{k}$ for all $k>0$. Therefore,
\begin{equation*}
    \Vert U_{\varepsilon}(t, \cdot)\Vert_{L^2}=\Vert u_{\varepsilon}(t,\cdot)-\Tilde{u}_{\varepsilon}(t,\cdot)\Vert_{L^2} \lesssim \varepsilon^{N},
\end{equation*}
for all $N>0$, which means that the very weak solution is unique.
\end{proof}

\subsection{Consistency}
Now we give the consistency result, which means that the very weak solution to the Cauchy problem (\ref{Equation}) converges in an appropriate norm, to the classical solution, when the latter exists.
\begin{thm}[Consistency]
Let $p\in L^{\infty}(\mathbb{R}^{d})$ be non-negative. Assume that $u_{0}\in H^{s}(\mathbb{R}^{d})$ for $s>0$, and let us consider the Cauchy problem
\begin{equation}
    \left\lbrace
    \begin{array}{l}
    iu_{t}(t,x)+(-\Delta)^{s} u(t,x) + p(x)u(t,x)=0 ,~~~(t,x)\in\left(0,T\right)\times \mathbb{R}^{d},\\
    u(0,x)=u_{0}(x). \label{Equation consistency}
    \end{array}
    \right.
\end{equation}
Let $(u_{\varepsilon})_{\varepsilon}$ be a very weak solution of (\ref{Equation consistency}). Then for any regularising families of the coefficient $p$ and the Cauchy data $u_0$, the net $(u_{\varepsilon})_{\varepsilon}$ converges in $L^{2}$ as $\varepsilon \rightarrow 0$ to the unique classical solution of the Cauchy problem (\ref{Equation consistency}).
\end{thm}

\begin{proof}
Let $u$ be the classical solution to
\begin{equation*}
    \left\lbrace
    \begin{array}{l}
    iu_{t}(t,x)+(-\Delta)^{s} u(t,x) + p(x)u(t,x)=0 ,~~~(t,x)\in\left(0,T\right)\times \mathbb{R}^{d},\\
    u(0,x)=u_{0}(x),
    \end{array}
    \right.
\end{equation*}
and let $(u_{\varepsilon})_{\varepsilon}$ its very weak solution. It satisfies
\begin{equation*}
    \left\lbrace
    \begin{array}{l}
    i\partial u_{\varepsilon}(t,x)+(-\Delta)^{s} u_{\varepsilon}(t,x) + p_{\varepsilon}(x)u_{\varepsilon}(t,x)=0 ,~~~(t,x)\in\left(0,T\right)\times \mathbb{R}^{d},\\
    u_{\varepsilon}(0,x)=u_{0,\varepsilon}(x).
    \end{array}
    \right.
\end{equation*}
Let us denote by $W_{\varepsilon}(t,x):=u(t,x)-u_{\varepsilon}(t,x)$. It solves
\begin{equation*}
    \left\lbrace
    \begin{array}{l}
    i\partial W_{\varepsilon}(t,x)+(-\Delta)^{s} W_{\varepsilon}(t,x) + p_{\varepsilon}(x)W_{\varepsilon}(t,x)=\eta_{\varepsilon}(t,x) ,~~~(t,x)\in\left(0,T\right)\times \mathbb{R}^{d},\\
    W_{\varepsilon}(0,x)=(u_{0}-u_{0,\varepsilon})(x),
    \end{array}
    \right.
\end{equation*}
where $\eta_{\varepsilon}(t,x):=(p_{\varepsilon}(x)-p(x))u(t,x)$. Using Duhamel's principle and similar arguments as in Theorem (\ref{Thm uniqueness}), we get the estimate
\begin{align*}
    \Vert W_{\varepsilon}(t, \cdot)\Vert_{L^2} & \lesssim \Vert u_{0}-u_{0,\varepsilon}\Vert_{L^2} + \int_{0}^{T}\Vert \eta_{\varepsilon}(s,\cdot)\Vert_{L^2} ds\\
    & \lesssim \Vert u_{0}-u_{0,\varepsilon}\Vert_{L^2} + \Vert p_{\varepsilon} - p\Vert_{L^{\infty}}\int_{0}^{T}\Vert u(s,\cdot)\Vert_{L^2} ds.
\end{align*}
When $\varepsilon\rightarrow 0$, the right hand side of the last inequality tends to $0$, since $\Vert p_{\varepsilon}-p\Vert_{L^{\infty}}\rightarrow 0$ and $\Vert u_{0}-u_{0,\varepsilon}\Vert_{L^2}\rightarrow 0$. Hence, the very weak solution converges to the classical one in $L^2$ as $\varepsilon\to0$.
\end{proof}

\begin{figure}[ht!]
\begin{minipage}[h]{0.31\linewidth}
\center{\includegraphics[scale=0.25]{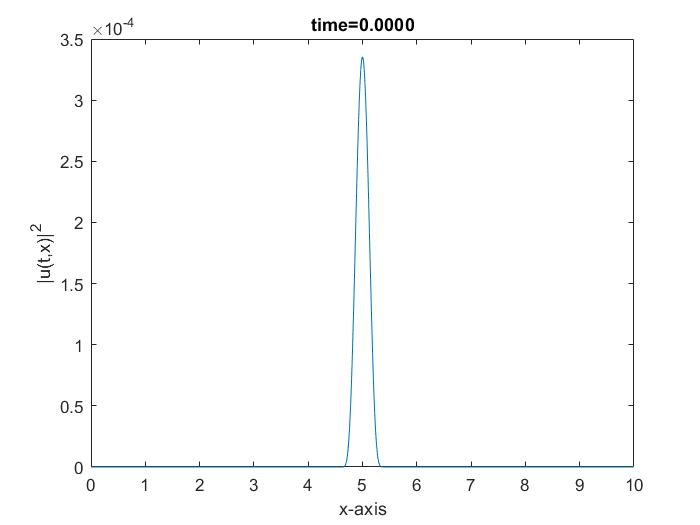}}
\end{minipage}
\hfill
\begin{minipage}[h]{0.31\linewidth}
\center{\includegraphics[scale=0.25]{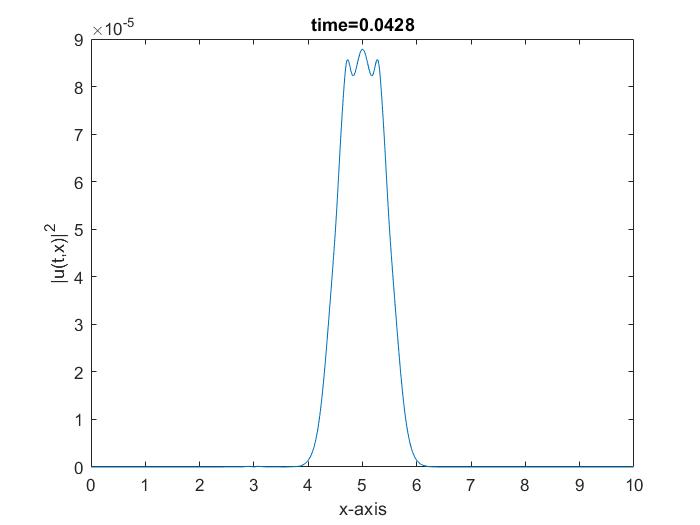}}
\end{minipage}
\hfill
\begin{minipage}[h]{0.31\linewidth}
\center{\includegraphics[scale=0.25]{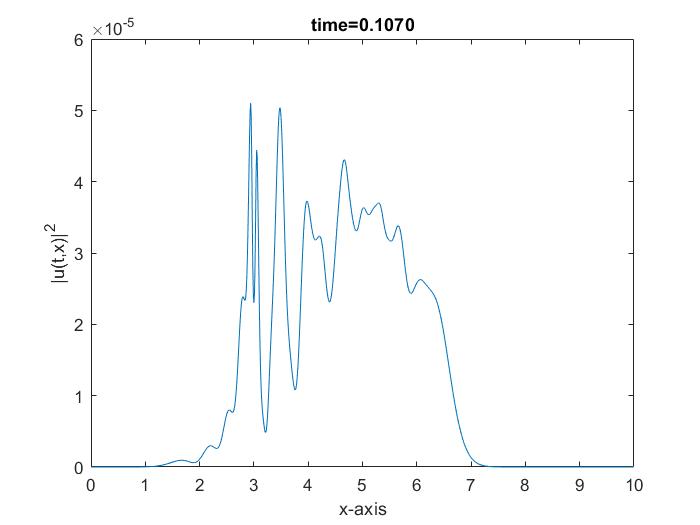}}
\end{minipage}
\hfill
\begin{minipage}[h]{0.31\linewidth}
\center{\includegraphics[scale=0.25]{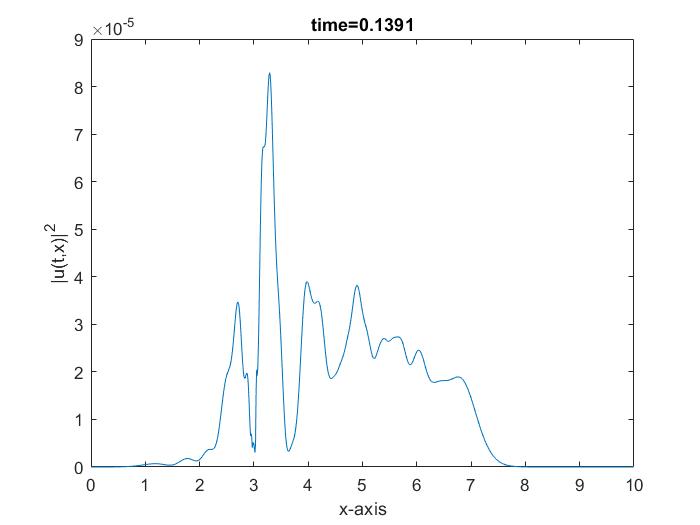}}
\end{minipage}
\hfill
\begin{minipage}[h]{0.31\linewidth}
\center{\includegraphics[scale=0.25]{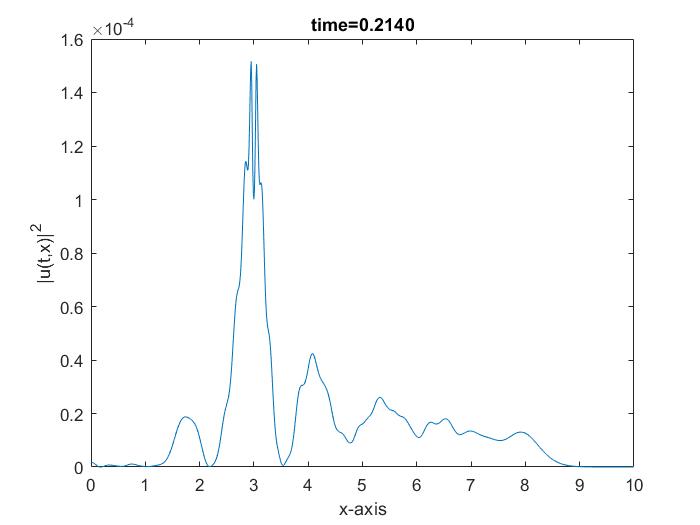}}
\end{minipage}
\hfill
\begin{minipage}[h]{0.31\linewidth}
\center{\includegraphics[scale=0.25]{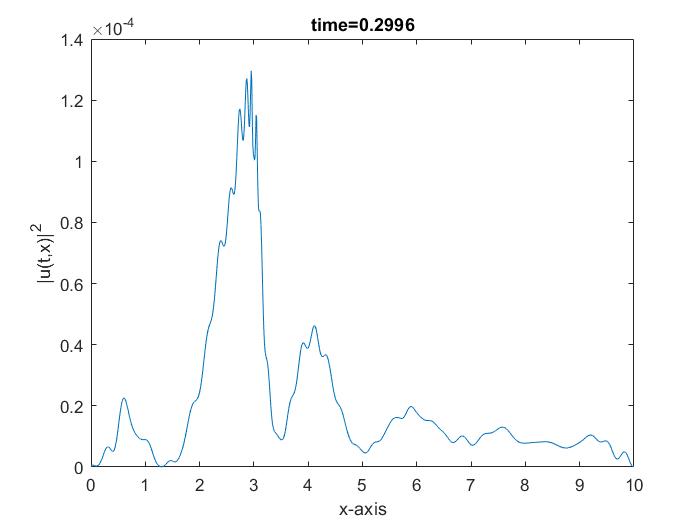}}
\end{minipage}
\caption{In these plots, we analyse behaviour of the solution of the Schr\"{o}dinger equation \eqref{RE-01} with a $\delta$-like potential. In the top left plot, the graphic of the position density of particles at the initial time is given. In the further plots, we draw the position density function $|u|^{2}$ at $t=0.0428, 0.1070, 0.1391, 0.2140, 0.2996$ for $\varepsilon=0.05$. Here, a $\delta$-like function with the support at point $3$ is considered.} 
\label{fig1}
\end{figure}

\begin{figure}[ht!]
\begin{minipage}[h]{0.31\linewidth}
\center{\includegraphics[scale=0.25]{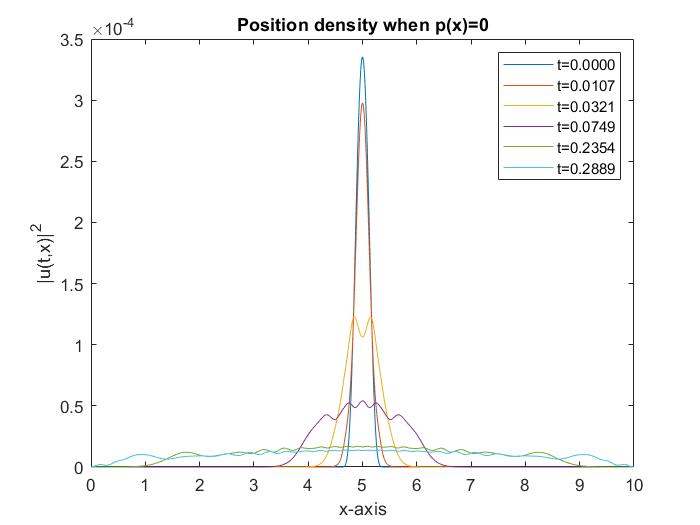}}
\end{minipage}
\hfill
\begin{minipage}[h]{0.31\linewidth}
\center{\includegraphics[scale=0.25]{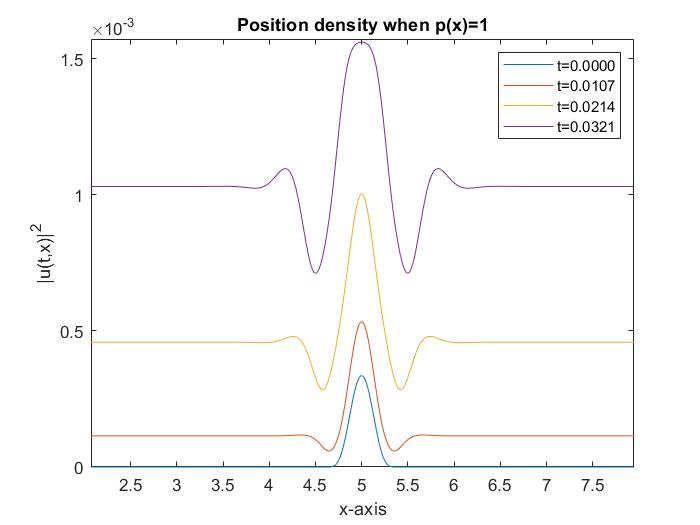}}
\end{minipage}
\hfill
\begin{minipage}[h]{0.31\linewidth}
\center{\includegraphics[scale=0.25]{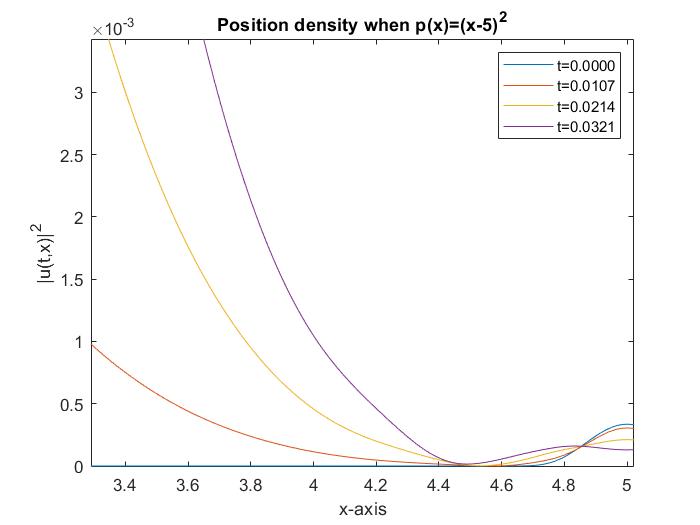}}
\end{minipage}
\caption{In these plots, we analyse the time evolution of the position density $|u|^{2}$ for different regular potentials. Here, the cases of the potentials with $p(x)=0, p(x)=1,$ and $p(x)=(x-5)^{2}$ are considered.} 
\label{fig2}
\end{figure}

\begin{figure}[ht!]
\begin{minipage}[h]{0.49\linewidth}
\includegraphics[scale=0.37]{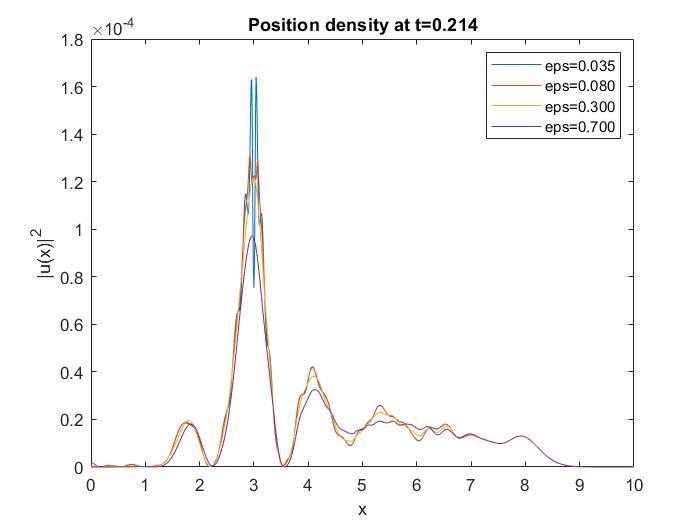}
\end{minipage}
\hfill
\begin{minipage}[h]{0.49\linewidth}
\includegraphics[scale=0.37]{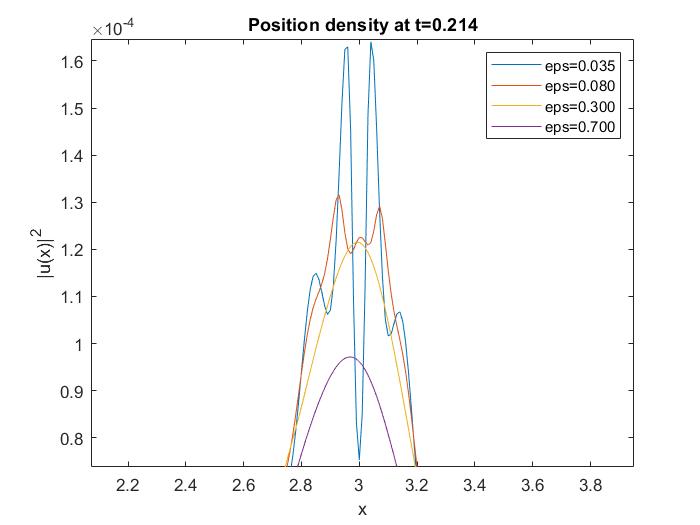}
\end{minipage}
\caption{In these plots, we analyse behaviour of the solution of the Schr\"{o}dinger equation \eqref{RE-01} with a $\delta$-like potential for different values of the parameter $\varepsilon$. Here, we compare the position density function of particles $|u|^{2}$ at $t=0.214$ for $\varepsilon=0.035, 0.080, 0.300, 0.800$. Here, the case of the potential with a $\delta$-like function behaviour with the support at point $3$ is considered.} 
\label{fig3}
\end{figure}

\begin{figure}[ht!]
\begin{minipage}[h]{0.49\linewidth}
\includegraphics[scale=0.37]{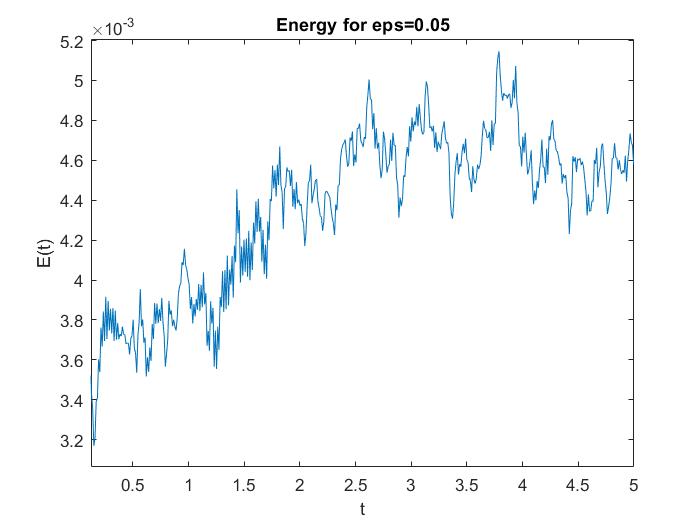}
\end{minipage}
\hfill
\begin{minipage}[h]{0.49\linewidth}
\includegraphics[scale=0.37]{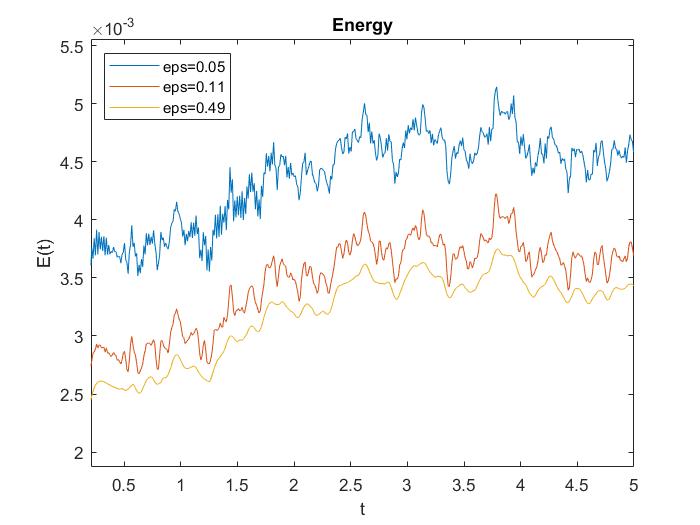}
\end{minipage}
\caption{In these plots, we compare the energy function $E(t)$ of the Schr\"{o}dinger equation \eqref{RE-01} corresponding to the $\delta$-potential case for $\varepsilon=0.05, 0.11, 0.49$.} 
\label{fig4}
\end{figure}

\begin{figure}[ht!]
\begin{minipage}[h]{0.49\linewidth}
\includegraphics[scale=0.37]{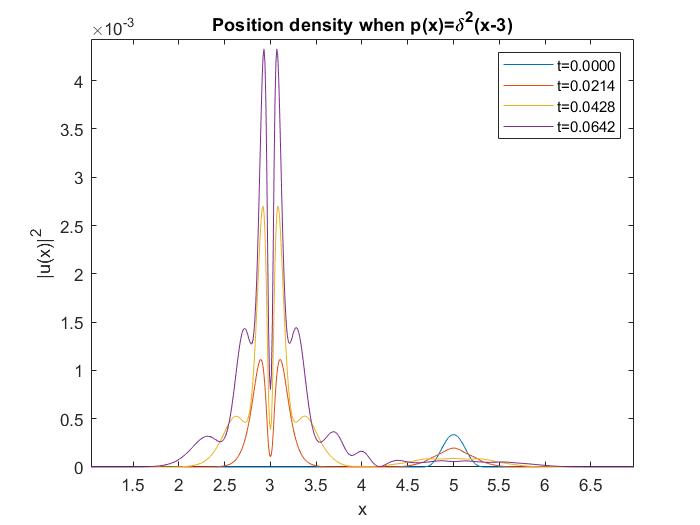}
\end{minipage}
\hfill
\begin{minipage}[h]{0.49\linewidth}
\includegraphics[scale=0.37]{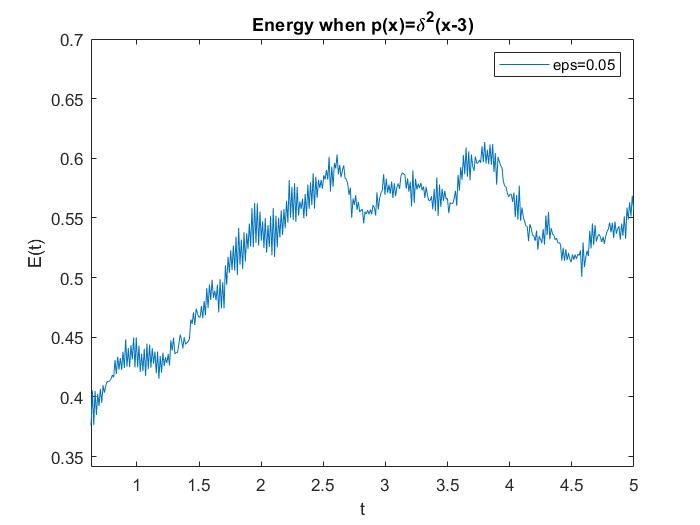}
\end{minipage}
\begin{minipage}[h]{0.49\linewidth}
\includegraphics[scale=0.37]{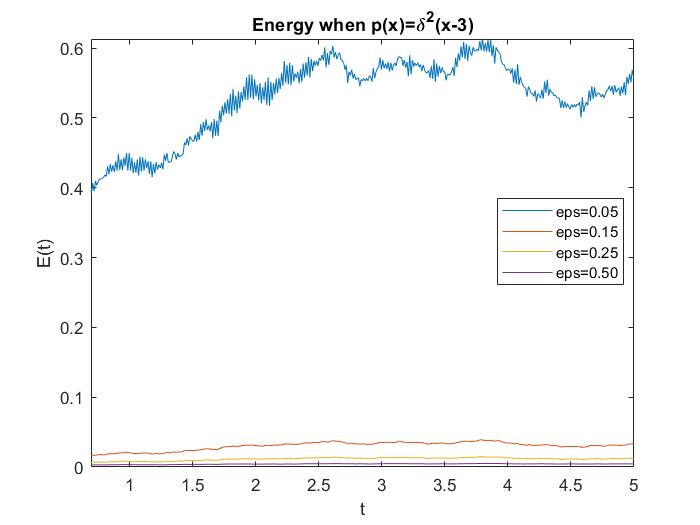}
\end{minipage}
\hfill
\begin{minipage}[h]{0.49\linewidth}
\includegraphics[scale=0.37]{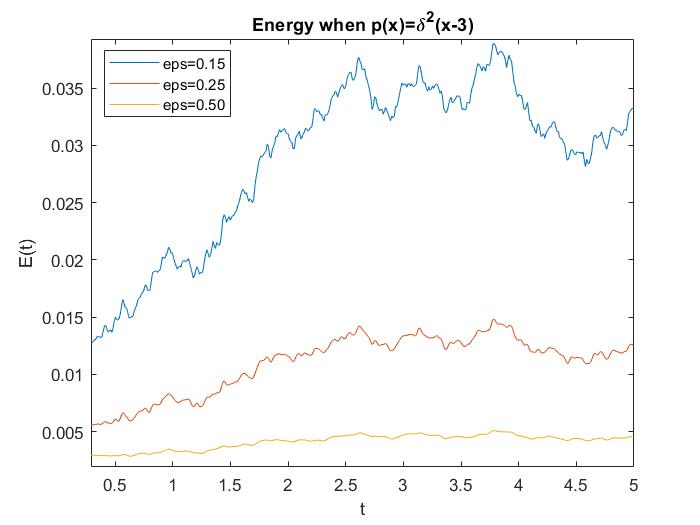}
\end{minipage}
\caption{In these plots, we analyse the solution of the Schr\"{o}dinger equation \eqref{RE-01} with a $\delta^{2}$-like potential. In the top left plot, we study the position density function $|u(t, x)|^{2}$ at $t=0.0000, 0.0214, 0.0428, 0.0642$ for $\varepsilon=0.05$. In further plots, we compare the energy function $E(t)$ of the Schr\"{o}dinger equation \eqref{RE-01} corresponding to the $\delta^{2}$-potential case for $\varepsilon=0.05, 0.15, 0.25, 0.50$. In the right-bottom plot, we compare the energy function for $\varepsilon=0.15, 0.25, 0.50$.} 
\label{fig5}
\end{figure}

\section{Numerical experiments}

In this Section, we do some numerical experiments. Let us analyse our problem by regularising a distributional potential $p(x)$ by a parameter $\varepsilon$. We define
$
p_\varepsilon (x):=(p\ast\varphi_\varepsilon)(x),
$
as the convolution with the mollifier
$\varphi_\varepsilon(x)=\frac{1}{\varepsilon} \varphi(x/\varepsilon),$
where
$$
\varphi(x)=
\begin{cases}
c \exp{\left(\frac{1}{x^{2}-1}\right)}, |x| < 1, \\
0, \,\,\,\,\,\,\,\,\,\,\,\,\,\,\,\,\,\,\,\,\,\,\,\,\,\,\,\,  |x|\geq 1,
\end{cases}
$$
with  $c \simeq 2.2523$ to have 
$
\int\limits_{-\infty}^{\infty}  \varphi(x)dx=1.
$
Then, instead of \eqref{Equation} we consider the regularised Cauchy problem for the 1D Schr\"{o}dinger equation
\begin{equation}\label{RE-01}
i\partial_{t}u_{\varepsilon}(t,x)-\partial^{2}_{x} u_{\varepsilon}(t,x)+ p_{\varepsilon}(x) u_{\varepsilon}(t,x) =0, \; (t,x)\in(0,T)\times\mathbb R,
\end{equation}
with the initial data $u_\varepsilon(0,x)=u_0 (x)$, for all $x\in\mathbb R.$ Here, we put
\begin{equation*}
\label{u_0}
u_0 (x)=
\begin{cases}
\exp{\left(\frac{1}{(x-5)^{2}-0.25}\right)}, \,\, |x-5| < 0.5, \\
0, \,\,\,\,\,\,\,\,\,\,\,\,\,\,\,\,\,\,\,\,\,\,\,\,\,\,\,\,\,\,\, \,\,\,\,\,\, \,\,\, |x-5| \geq 0.5.
\end{cases}
\end{equation*}
Note that ${\rm supp }\, u_0\subset[4.5, 5.5]$.

Here, we consider the following cases when potential is a regular function: $p(x)=0$, $p(x)=1$, and $p(x)=(x-5)^{2}$; when potential is a singular function: $p(x)=\frac{1}{30}\delta(x-3)$ with $p_{\varepsilon}(x)=\frac{1}{30}\varphi_\varepsilon(x-3)$ and $p(x)=\frac{1}{30}\delta^{2}(x-3)$ in the sense $p_{\varepsilon}(x)=\frac{1}{30}\varphi^{2}_\varepsilon(x-3)$, where $\delta$ denoting the standard Dirac's delta-distribution.

In Figure \ref{fig1}, we analyse behaviour of the solution of the Schr\"{o}dinger equation \eqref{RE-01} with a $\delta$-like potential. In the top left plot, the graphic of the position density of particles at the initial time is given. In the further plots, we draw the position density function $|u|^{2}$ at $t=0.0428, 0.1070, 0.1391, 0.2140, 0.2996$ for $\varepsilon=0.05$. Here, a $\delta$-like function with the support at point $3$ is considered.
We observe that a delta-function potential causing an accumulating of particles phenomena in the place of the support of the singularity.

In Figure \ref{fig2}, we analyse the time evolution of the position density for different regular potentials. Here, the cases of the potentials with $p(x)=0, p(x)=1,$ and $p(x)=(x-5)^{2}$ are considered.

In Figure \ref{fig3}, we analyse behaviour of the solution of the Schr\"{o}dinger equation \eqref{RE-01} with a $\delta$-like potential for different values of the parameter $\varepsilon$. Here, we compare the position density function of particles $|u|^{2}$ at $t=0.214$ for $\varepsilon=0.035, 0.080, 0.300, 0.800$. Here, the case of the potential with a $\delta$-like function behaviour with the support at point $3$ is considered. Here, we can see that the numerical simulations of the regularised equation \eqref{RE-01} are stable under the changing of the values of the parameter $\varepsilon$.

In Figure \ref{fig4}, we compare the energy function 
\begin{equation}
\label{Energy}
    E(t)=\Vert \nabla u(t,\cdot)\Vert_{L^2}^{2} + \Vert p^{\frac{1}{2}}(\cdot)u(t,\cdot)\Vert_{L^2}^{2}.
\end{equation}
of the Schr\"{o}dinger equation \eqref{RE-01} corresponding to the $\delta$-potential case for different values of the parameter $\varepsilon$. Simulations show that $E(t)\approx E(0)$ for $t>0$.

In Figure \ref{fig5}, we analyse the solution of the Schr\"{o}dinger equation \eqref{RE-01} with a $\delta^{2}$-like potential. In the left plot, we study the position density function $|u(t, x)|^{2}$ at $t=0.0000, 0.0214, 0.0428, 0.0642$ for $\varepsilon=0.05$. In the right plot, we compare the energy function $E(t)$ of the Schr\"{o}dinger equation \eqref{RE-01} corresponding to the $\delta^{2}$-potential case for $\varepsilon=0.05, 0.15, 0.25, 0.50$. From these plots we conclude that thanks to the concept of the very weak solution the studying of the processes in physics are possible despite the impossibility of the multiplication of the distributions in the theory of distributions.

\begin{rem}
By analysing these cases, from Figures \ref{fig4} and \ref{fig5} we see that the energy function $E(t)$ given by \eqref{Energy} satisfies $E(t)\approx E(0)$ for $t>0$. Moreover, it is observed that $E(t)$ depends on $\varepsilon$ by confirming the theory, that is, $E(t)=E_{\varepsilon}(t)$. From the bottom plots of Figure \ref{fig5} we observe that the energy $E(t)$ of the Schr\"{o}dinger equation \eqref{RE-01} with a  $\delta^{2}$-like potential corresponding to the case $\varepsilon=0.5$ is increased around $200$ times as $\varepsilon$ is decreased $10$ times by justifying the theoretical part.
\end{rem}

\begin{rem}
From the behaviours of the density function $|u(t, x)|^{2}$ of the Schr\"{o}dinger equation \eqref{RE-01} corresponding to the cases of $\delta$-like and $\delta^{2}$-like potentials, namely, from the left plot of Figure \ref{fig3} and the upper--left plot of Figure \ref{fig5} we observe a "splitting of the strong singularity" effect. Explanation of this phenomena is still an open question from the theoretical point of view.
\end{rem}

A second order in time and in space Crank-Nicolson scheme is used for the numerical analysis of the equation \eqref{RE-01}. All numerical computations are made in C++ by using the sweep method. In above numerical simulations, we use the Matlab R2018b. For all simulations we take $\Delta t=0.0107$, $\Delta x=0.01.$

\subsection{Conclusion} The theoretical and numerical analysis conducted in this paper showed that numerical methods work well in situations where a rigorous mathematical formulation of the problem is difficult in the framework of the classical theory of distributions. The ideology of very weak solutions eliminates this difficulty in the case of the terms with multiplication of distributions. In particular, in the case of the Schr\"{o}dinger equation, we see that a delta-function potential causing an effect of accumulating particles in the place of the support of the singularity.

Numerical simulations have shown that the idea of very weak solutions suit nice to numerical modelling. Moreover, using the theory of very weak solutions, we are able to deal with the uniqueness of numerical solutions of partial differential equations with coefficients of higher order singularity in some appropriate sense.

\section*{Acknowledgments}
The authors were supported in parts by the FWO Odysseus 1 grant G.0H94.18N: Analysis and Partial Differential Equations. Michael Ruzhansky was supported in parts by the EPSRC Grant EP/R003025/2. Arshyn Altybay was supported in parts by the MESRK Grant AP08052028 of the Science Committee of the Ministry of Education and Science of the Republic of Kazakhstan. Mohammed Sebih was supported by the Algerian Scholarship P.N.E. 2018/2019 during his visit to the University of Stuttgart and Ghent University. Also, Mohammed Sebih thanks Professor Jens Wirth and Professor Michael Ruzhansky for their warm hospitality.


\end{document}